\documentclass[reqno]{amsart}
\usepackage{euscript}

\theoremstyle{plain}
\numberwithin{equation}{section}
\newtheorem{theorem}{Theorem}
\newtheorem{proposition}{Proposition}
\newtheorem{lemma}{Lemma}
\newtheorem{corollary}{Corollary}
\theoremstyle{remark}
\newtheorem{remark}{Remark}
\newtheorem{example}{Example}
\renewcommand{\epsilon}{\varepsilon}
\renewcommand{\phi}{\varphi}

\DeclareMathOperator{\Ker}{Ker}
\DeclareMathOperator{\esssup}{\mathrm{ess\,sup}}
\def\N{\mathbb{N}}

\def\R{\mathbb{R}}

\def\Id{\text{\rm Id}}

\begin{document}

\title{Admissibility and general dichotomies for evolution families}

\author{Davor Dragi\v cevi\'c}
\address{Department of Mathematics, University of Rijeka, Radmile Matej\v{c}i\'{c} 2, 51000 Rijeka, Croatia}
\email{ddragicevic@math.uniri.hr}

\author{Nevena Jur\v cevi\' c Pe\v cek}
\address{Department of Mathematics, University of Rijeka, Radmile Matej\v{c}i\'{c} 2, 51000 Rijeka, Croatia}
\email{njurcevic@math.uniri.hr}

\author{Nicolae Lupa}
\address{Department of Mathematics, Politehnica University of Timi\c{s}oara,
Pia\c ta Victoriei 2, 300006 Timi\c soara, Romania}
\email{nicolae.lupa@upt.ro}

\begin{abstract}
 For an arbitrary noninvertible evolution family on the half-line and for  $\rho \colon [0, \infty)\to [0, \infty)$ in a large class of rate functions, we consider the notion of a $\rho$-dichotomy with respect to a family of norms and characterize it in terms of two admissibility conditions. In particular, our results are applicable to  exponential as well as polynomial dichotomies with respect to a family of norms. As a nontrivial application of our work, we establish the robustness of general nonuniform dichotomies.
\end{abstract}

\maketitle
\section{Introduction}
Among many methods used to study the asymptotic  behavior of nonautonomous systems, one of the most famous is the so-called admissibility method. This line of research in the context of differential equations has a long history that goes back to the pioneering work of Perron~\cite{Pe}.
The main  idea of Perron's work was  to characterize
the asymptotic properties of the linear differential equation
\[
{\dot x}(t) = A(t)x(t), \quad t \in {\mathbb J},
\]
in terms of the (unique) solvability in $O({\mathbb J}, X)$  of the equation
\[
{\dot x}(t) = A(t)x(t)  + f(t), \quad t \in {\mathbb J},
\]
for each test function  $f\in I({\mathbb J}, X)$, where ${\mathbb J}\in \{[0, \infty), \R\}$. Here  $X$ is a Banach space,  while $I({\mathbb J}, X)$ - the input-space and $O({\mathbb J}, X)$ - the output space  are suitably constructed function spaces.
The milestones of this theory were grounded in the sixtieth in the remarkable works of Massera and  Sch\"affer~\cite{MaS-MA-1960-I,MaS-MA-1960-IV,MS} and respectively in the seventies in the outstanding monographs of Coppel \cite{Co-78} and Daleck\u \i i and Kre\u \i n \cite{DaK-1974}.

Since then various authors obtained valuable contributions to this line of the research. For the results dealing with characterizations of \emph{uniform} exponential behavior in terms of appropriate admissibility properties, we refer to the works of Huy~\cite{Huy}, Latushkin, Randolph and Schnaubelt~\cite{LRS}, Van Minh,  R\"{a}biger and Schnaubelt~\cite{Mi.Ra.Sch}, Van Minh and Huy~\cite{MH},  Preda, Pogan and Preda~\cite{Preda, Preda2} as well as  Sasu and  Sasu~\cite{Sasu1, Sasu3, Sasu2, Sasu4, Sasu5}. For contributions dealing with various concepts of nonuniform exponential behavior, we refer to~\cite{Ba.Dr.Va.1, Ba.Dr.Va.2, Lu.Po, Me.Sa.Sa, PM, Sa.Sa.1,Zh.Lu.Zhang} and references therein. For a detailed description of this line of the research, we refer to~\cite{book}.

We point out  that all the above works deal with \emph{exponential} behavior. Although this type of behavior has a somewhat privileged role due to its connections with the hyperbolic smooth dynamics, it is certainly not the only type of behavior that appears in the qualitative study of nonautonomous differential equations.
To the best of our knowledge, the study of dichotomies with not necessarily exponential rates of expansion and contraction was initiated by Muldowney~\cite{Mul} and Naulin and Pinto~\cite{NP}. More recently, in the context of nonuniform asymptotic behavior such dichotomies have been studied by Barreira and Valls  \cite{Ba.Va.1,Ba.Va.3}
and Bento and Silva \cite{BS1,BS2}. A special emphasis was devoted to the so-called polynomial dichotomies~\cite{Ba.Va.4,Be.Si}.
A complete characterization of polynomial dichotomies in terms of  admissibility for evolution families was obtained by Dragi\v cevi\' c~\cite{Da.1} (see also~\cite{d} for  related results in the case of discrete time) by building on the work of Hai~\cite{H}, who considered polynomial stability.

The main objective of the present paper is to obtain similar results to that in~\cite{Da.1} but for a much wider class of dichotomies. More precisely, for a large class of rate functions $\rho \colon [0, \infty) \to [0, \infty)$, we introduce the notion of a $\rho$-dichotomy with respect to a family of norms. 
We then obtain a complete characterization of this concept in terms of  appropriate admissibility conditions. We point out that our results are  new even in the particular case of \emph{uniform}
$\rho$-dichotomies. Indeed, although the proofs use the somewhat standard techniques, the major task accomplished in the present paper was to formulate appropriate  admissibility conditions for the general dichotomies we study. In addition, the obtained results are new even for the class of polynomial dichotomies since
in comparison to~\cite{Da.1}, we don't require that our evolution family exhibits polynomial bounded growth property. Consequently, we need to impose two admissibility conditions (rather than just one as in~\cite{Da.1}) to characterize polynomial dichotomies. We stress that in the present paper we also use different
admissibility spaces from those in~\cite{Da.1}.

The paper is organized as follows. In Section~\ref{P} we introduce the class of dichotomies we study as well as input and output spaces we are going to use. 
In Section~\ref{DA}, we show that the existence of $\rho$-dichotomies yields two types of admissibility properties. Then, in Section~\ref{AD} we obtain a converse results by showing that those admissibility properties imply the existence of a $\rho$-dichotomy. Finally, in Section~\ref{R} we apply those results to establish the robustness of $\rho$-dichotomies.

\section{Preliminaries}\label{P}

\subsection{Generalized dichotomies}

Let $X=(X, \lVert \cdot \rVert)$ be an arbitrary Banach space  and  let $\mathcal{B}(X)$ be the Banach algebra of all bounded linear operators on $X$.
A family $\mathcal{T}=\{T(t,s)\}_{t\ge s \ge 0}$ of operators  in $\mathcal{B}(X)$ is said to be an \emph{evolution family} on $X$ if the following properties hold:
\begin{itemize}
\item $T(t,t)=\Id$, for $t\ge 0$;
\item $T(t,s)T(s, \tau)=T(t, \tau)$, for $t\ge s\ge \tau\ge 0$;
\item for all $s\geq 0$ and  $x\in X$ the mapping $t\mapsto T(t,s)x$ is continuous on $[s, \infty)$ and the mapping $t\mapsto T(s,t)x$ is continuous on $[0,s]$.
\end{itemize}

In this paper we always assume that  $\mathcal{T}=\{T(t,s)\}_{t\ge s \ge 0}$ is an evolution family on $X$ and
let  $\rho \colon [0, \infty )\to [0, \infty)$ be a strictly increasing function of class $C^1$ such that $$\rho(0)=0 \text{ and }\lim\limits_{t\to \infty} \rho(t)=\infty.$$  In particular, observe that $\rho$ is bijective.
Furthermore, assume that $\{\lVert \cdot \rVert_t\}_{t\ge 0}$ is a family of norms on $X$ such that:
\begin{itemize}
\item there exist $C>0$ and $\epsilon \ge 0$ with
\begin{equation}\label{427}
\lVert x\rVert \le \lVert x\rVert_t \le Ce^{\epsilon\rho(t)} \lVert x\rVert, \quad \text{for $x\in X$ and $t\ge 0$;}
\end{equation}
\item the mapping $t\mapsto \lVert x\rVert_t$ is continuous for each $x\in X$.
\end{itemize}

We say that the evolution family $\mathcal{T}$ admits a \emph{$\rho$-dichotomy} with respect to the family of norms $\lVert \cdot \rVert_t$, $t\ge 0$, if there exists
a family $\{P(t)\}_{t\geq 0}$ of projections  on $X$ satisfying
\begin{equation}\label{pro}
T(t,s)P(s)=P(t)T(t,s), \text{ for } t\ge s\ge 0,
\end{equation}
such that 
\begin{equation}\label{pro.inv}
T(t,s)\rvert_{\Ker P(s)} \colon \Ker P(s) \to \Ker P(t) \text{ is invertible for all $t\ge s\ge 0$},
\end{equation}
and there exist constants
$\lambda,D >0$  such that:
\begin{itemize}
\item for $x\in X$ and $t\ge s\ge 0$,
\begin{equation}\label{d1}
\lVert T(t,s)P(s)x\rVert_t \le De^{-\lambda (\rho(t)-\rho(s))}\lVert x\rVert_s;
\end{equation}
\item for $x\in X$ and $0\le t\le s$,
\begin{equation}\label{d2}
\lVert T(t,s)(\Id-P(s))x\rVert_t \le De^{-\lambda (\rho(s)-\rho(t))}\lVert x\rVert_s,
\end{equation}
where
\[
T(t,s):=\bigg{(} T(s,t )\rvert_{\Ker P(t)} \bigg{)}^{-1} \colon \Ker P(s) \to \Ker P(t),
\]
for $0\le t\le s$.
\end{itemize}

In the following we recall the concept of \emph{$\rho$-nonuniform exponential dichotomy} for evolution families (see \cite{Ba.Va.1,Ba.Va.3}) and establish its connection with the notion of $\rho$-dichotomy with respect to a family of norms.
An evolution family $\mathcal{T}$
is said to admit a \emph{$\rho$-nonuniform exponential dichotomy} if there exists a family $\{P(t)\}_{t\geq 0}$ of projections  on $X$ satisfying \eqref{pro} and \eqref{pro.inv},
and there exist constants $\lambda,D>0$ and $\epsilon\geq 0$ such that
\begin{equation}\label{dich.s}
\| T(t,s) P(s)\|\leq De^{-\lambda(\rho(t)-\rho(s))+\epsilon\rho(s)}, \text{ for } t\geq s\geq 0,
\end{equation}
and
\begin{equation}\label{dich.u}
\| T(t,s)(\Id- P(s))\|\leq De^{-\lambda(\rho(s)-\rho(t))+\epsilon\rho(s)}, \text{ for } 0\leq t\leq s.
\end{equation}

The concept of $\rho$-nonuniform exponential dichotomy includes as a special case the usual  \emph{exponential behavior} when $\rho(t)=t$. Also, for $\rho(t)=\ln(t+1)$ we obtain the concept of \emph{nonuniform polynomial dichotomy} introduced independently by Barreira and Valls \cite{Ba.Va.4} and Bento and Silva  \cite{Be.Si}, and more general for $\rho(t)=\int_{0}^t \mu(t) dt,$ where $\mu:[0,\infty)\to (0,\infty)$ is a continuous function such that $\lim\limits_{t\to \infty} \int_{0}^t \mu(t) dt=\infty,$ we obtain the nonuniform version of the \emph{generalized dichotomy} in the sense of  Muldowney \cite{Mul}.

\begin{proposition}\label{prop.equiv}
The following statements are equivalent:
\begin{enumerate}
\item $\mathcal{T}$ admits a $\rho$-nonuniform exponential dichotomy;
\item $\mathcal{T}$ admits a $\rho$-dichotomy with respect to a family of norms $\|\cdot\|_t$, $t\geq 0$ such that $t\mapsto \lVert x\rVert_t$ is continuous for each $x\in X$.
\end{enumerate}
\end{proposition}
\begin{proof}
Assume that $\mathcal{T}$ admits a $\rho$-nonuniform exponential dichotomy. For $t\geq 0$ and $x\in X$, set
\[
\|x\|_t=\sup\limits_{\tau \geq t} e^{\lambda (\rho(\tau)-\rho(t))}\|T(\tau,t)P(t)x\|+\sup\limits_{\tau\in[0,t]} e^{\lambda (\rho(t)-\rho(\tau))}\|T(\tau,t)(\Id-P(t))x\|.
\]
A simple computation shows that \eqref{427} holds for $C=2D$. Moreover, by repeating the arguments in the proof of~\cite[Proposition 5.6]{book}, one can show that  $t\mapsto \lVert x\rVert_t$ is continuous for each $x\in X$. Furthermore, for $t\geq s\geq 0$ and $x\in X$ we have
\begin{align*}
\| T(t,s)P(s)x\|_t&=\sup\limits_{\tau \geq t} e^{\lambda (\rho(\tau)-\rho(t))}\|T(\tau,s)P(s)x\|\\
&=\sup\limits_{\tau \geq t}e^{-\lambda(\rho(t)-\rho(s))} e^{\lambda (\rho(\tau)-\rho(s))}\|T(\tau,s)P(s)x\|\\
&\leq  e^{-\lambda(\rho(t)-\rho(s))} \sup\limits_{\tau \geq s} e^{\lambda (\rho(\tau)-\rho(s))}\|T(\tau,s)P(s)x\|\\
&\leq  e^{-\lambda(\rho(t)-\rho(s))} \| x\|_s,
\end{align*}
and thus \eqref{d1} holds. Similarly, one can prove \eqref{d2}. Hence, the evolution family $\mathcal{T}$ admits a
$\rho$-dichotomy with respect to the family of norms $\|\cdot\|_t$, $t\geq 0$, defined above.

Conversely, assume that $\mathcal{T}$  admits a $\rho$-dichotomy with respect to a family of norms $\|\cdot\|_t$ on $X$ satisfying
\eqref{427} for some $C>0$ and $\epsilon\geq 0$. For $t\geq s\geq 0$ and $x\in X$ we have
\begin{align*}
\| T(t,s)P(s)x \| &\leq \| T(t,s)P(s)x \|_t\\
&\leq D e^{-\lambda(\rho(t)-\rho(s))}\|x\|_s\\
&\leq D C e^{\epsilon \rho(s)} e^{-\lambda(\rho(t)-\rho(s))}\|x\|,
\end{align*}
and thus \eqref{dich.s} holds.
Similarly, one can show  \eqref{dich.u}. Therefore, the evolution family $\mathcal{T}$ admits a $\rho$-nonuniform exponential dichotomy.
\end{proof}

\subsection{Admissible spaces}

Let $Y_1$ be the space of all Bochner measurable functions $x\colon [0, \infty)\to X$ such that
\[
\lVert x\rVert_1:=\int_0^\infty \lVert x(t)\rVert_t\, dt<\infty,
\]
identifying functions that are equal Lebesque-almost everywhere.
It is easy to show that  $(Y_1, \lVert \cdot \rVert_1)$ is  a Banach space (see  \cite[Theorem 1]{Ba.Dr.Va.1}). Moreover, consider the space $Y_\infty$ of all continuous functions $x\colon [0, \infty) \to X$ such that
\[
\lVert x\rVert_\infty:=\sup\limits_{t\ge 0}\lVert x(t)\rVert_t <\infty.
\]
One can easily prove that  $(Y_\infty, \lVert \cdot \rVert_\infty)$ is a Banach space. For a closed subspace $Z\subset X$,  $Y_\infty^Z$ is the space of all $x\in Y_\infty$ such that $x(0)\in Z$. Obviously, $Y_\infty^Z$ is a closed subspace of $Y_\infty$, therefore it is also a Banach space.

We consider another Banach function space  $(Y_\infty', \lVert \cdot \rVert_\infty')$, which consist of all Bochner measurable functions $x\colon [0, \infty) \to X$ such that
\[
\lVert x\rVert_\infty':=\esssup\limits_{t\ge 0}\lVert x(t)\rVert_t <\infty,
\]
where $\esssup$ is taken with respect to the Lebesgue measure on $[0, \infty)$.

\section{From dichotomy to admissibility}\label{DA}

In this section we show that the existence of a $\rho$-dichotomy with respect to a family of norms for an evolution family $\mathcal{T}=\{T(t,s)\}_{t\ge s \ge 0}$  yields the admissibility of the pairs
$\left(Y_\infty^Z,Y_1\right)$, $\left(Y_\infty^Z,Y_\infty'\right)$ for a certain closed subspace $Z\subset X$.

\begin{proposition}\label{p1}
Assume that the evolution family $\mathcal{T}$ admits a $\rho$-dichotomy with respect to a family of norms $\lVert \cdot \rVert_t$, $t\ge 0$, and set $Z=\Ker P(0)$. Then, for each $y\in Y_1$ there exists a unique $x\in Y_\infty^Z$ such that
\begin{equation}\label{532}
x(t)=T(t,s)x(s)+\int_s^t T(t, \tau)y(\tau)\, d\tau, \quad \text{for $t\ge s\ge 0$.}
\end{equation}
\end{proposition}

\begin{proof}
Take an arbitrary $y\in Y_1$. For $t\ge 0$, set
\[
x(t)=\int_0^t T(t,s)P(s)y(s)\, ds-\int_t^\infty T(t,s)(\Id-P(s))y(s)\, ds.
\]
It follows from~\eqref{d1} and~\eqref{d2} 
that
\[
\begin{split}
\lVert x(t)\rVert_t &\le \int_0^t \lVert  T(t,s)P(s)y(s)\rVert_t\, ds +\int_t^\infty \lVert T(t,s)(\Id-P(s))y(s)\rVert_t\, ds \\
&\le D \int_0^t  e^{-\lambda (\rho(t)-\rho(s))} \lVert y(s)\rVert_s \, ds +D\int_t^\infty e^{-\lambda (\rho(s)-\rho(t))}\lVert y(s)\rVert_s\, ds \\
&\le D\int_0^t \lVert y(s)\rVert_s \, ds +\int_t^\infty \lVert y(s)\rVert_s\, ds
= D\,\| y\|_1,
\end{split}
\]
for every $t\ge 0$, and thus $x\in Y_\infty$. On the other hand, it is easy to check that $x(0)\in Z$. Therefore, $x\in Y_\infty^Z$. 
Moreover, for $t\ge s\geq 0$ we have
\[
\begin{split}
x(t)-T(t,s)x(s) &=\int_0^tT(t,\tau)P(\tau)y(\tau)\, d\tau -T(t,s)\int_0^sT(s,\tau)P(\tau)y(\tau)\, d\tau \\
&\phantom{=}-\int_t^\infty T(t,\tau)(\Id-P(\tau))y(\tau)\, d\tau \\
&\phantom{=}+T(t,s)\int_s^\infty T(s,\tau)(\Id-P(\tau))y(\tau)\, d\tau \\
&=\int_s^t T(t,\tau)P(\tau)y(\tau)\, d\tau +\int_s^t T(t, \tau)(\Id-P(\tau))y(\tau)\, d\tau \\
&=\int_s^tT(t, \tau)y(\tau)\, d\tau,
\end{split}
\]
and therefore we conclude that~\eqref{532} holds.
In order to establish the uniqueness, it is sufficient to consider the case when $y=0$. Let $x\in Y_\infty^Z$  such that $$x(t)=T(t,s)x(s), \text{ for } t\ge s\ge 0.$$  Then, from \eqref{d2}  we have
\begin{align*}
\lVert x(0)\rVert_0 & = \| (\Id -P(0))x(0)\|_0  =\lVert T(0, t)(\Id -P(t))x(t)\rVert_0\\
&\le De^{-\lambda \rho(t)}\lVert x(t)\rVert_t \\
&\le  De^{-\lambda \rho(t)} \lVert x\rVert_\infty,
\end{align*}
for every $t\geq 0$.
Passing to the limit when $t\to \infty$, we conclude that $x(0)=0$, which implies that $x=0$.
\end{proof}

\begin{proposition}\label{p2}
Assume that the evolution family $\mathcal{T}$ admits a $\rho$-dichotomy with respect to a family of norms $\lVert \cdot \rVert_t$, $t\ge 0$, and set $Z=\Ker P(0)$. Then, for each $y\in Y_\infty'$ there exists a unique $x\in Y_\infty^Z$ such that
\begin{equation}\label{617}
x(t)=T(t,s)x(s)+\int_s^t \rho'(\tau)T(t, \tau)y(\tau)\, d\tau, \quad \text{for $t\ge s\ge 0$.}
\end{equation}
\end{proposition}

\begin{proof}
Take $y\in Y_\infty'$. For $t\ge 0$,  set
\[
x(t)=\int_0^t \rho'(s)T(t,s)P(s)y(s)\, ds-\int_t^\infty \rho'(s) T(t,s)(\Id-P(s))y(s)\, ds.
\]
It follows from~\eqref{d1} and~\eqref{d2} that
\[
\begin{split}
\lVert x(t)\rVert_t &\le \int_0^t \rho'(s) \lVert  T(t,s)P(s)y(s)\rVert_t\, ds +\int_t^\infty \rho' (s)\lVert T(t,s)(\Id-P(s)) y(s)\rVert_t\, ds \\
&\le  D \int_0^t  \rho'(s) e^{-\lambda (\rho(t)-\rho(s))} \lVert y(s)\rVert_s \, ds +D\int_t^\infty \rho'(s)e^{-\lambda (\rho(s)-\rho(t))}\lVert y(s)\rVert_s\, ds \\
&\le D\lVert y\rVert_\infty' \bigg{(}\int_0^t  \rho'(s) e^{-\lambda (\rho(t)-\rho(s))} \, ds+\int_t^\infty \rho'(s)e^{-\lambda (\rho(s)-\rho(t))}\, ds \bigg{)} \\
&\le \frac{2D}{\lambda}\lVert y\rVert_\infty', \text{ for every $t\ge 0$.}
\end{split}
\]
Since $x(0) \in Z$, we conclude that $x\in Y_\infty^Z$.
A simple computation shows that \eqref{617} holds.
The uniqueness part can be established as in the proof of Proposition~\ref{p1}.
\end{proof}

\section{From admissibility to dichotomy}\label{AD}

The aim of this section is to prove that the admissibility of the pairs
$\left(Y_\infty^Z,Y_1\right)$,  $\left(Y_\infty^Z,Y_\infty'\right)$ for a closed subspace $Z\subset X$ yields the existence of a $\rho$-dichotomy with respect to the family of norms $\{\lVert \cdot \rVert_t\}_{t\ge 0}$.
More precisely, our goal is to establish the following result.

\begin{theorem}\label{th.ad.dich}
Assume that there exists a closed subspace $Z\subset X$ such that:
\begin{enumerate}
\item[(i)]  for each $y\in Y_1$ there exists a unique $x\in Y_\infty^Z$ satisfying \eqref{532};
\item[(ii)] for each $y\in Y_\infty'$ there exists a unique $x\in Y_\infty^Z$  satisfying \eqref{617}.
\end{enumerate}
Then, the evolution family $\mathcal{T}$ admits a $\rho$-dichotomy with respect to the family of norms $\lVert \cdot \rVert_t$, $t\ge 0$.
\end{theorem}

\begin{proof}
Let $$T_Z:\mathcal D(T_Z)\subset Y_\infty^Z \to Y_1, \quad T_Zx=y,$$
where
$$\mathcal D(T_Z)=\left\{ x\in Y_\infty^Z:\; \text{ there exists } y\in Y_1 \text{ satisfying } \eqref{532}\right\}.$$
Furthermore, let
$$T_Z':\mathcal D(T_Z')\subset Y_\infty^Z \to Y_\infty', \quad T_Z'x=y,$$
where
$$\mathcal D(T_Z')=\left\{ x\in Y_\infty^Z:\; \text{ there exists } y\in Y_\infty' \text{ satisfying } \eqref{617}\right\}.$$

\medskip

\begin{lemma}
The operators  $T_Z \colon \mathcal D(T_Z) \to Y_1$,  $T_Z' \colon \mathcal D(T_Z') \to Y_\infty'$ are well-defined, linear  and closed.
\end{lemma}
\begin{proof}[Proof of the lemma]
Assume that $x\in Y_\infty^Z$ and $y_1, y_2\in Y_1$ such that
\[
x(t)=T(t,\tau)x(\tau)+\int_\tau^t  T(t,s)y_i(s)\, ds,
\]
for $t\ge \tau \ge 0$ and $i\in \{1, 2\}$. Hence,
\[
\int_\tau^t  T(t,s)(y_1(s)-y_2(s))\, ds=0, \text{ for $t> \tau \ge 0$.}
\]
Dividing by $t-\tau$ and letting $t-\tau \to 0$, it follows from the Lebesgue differentiation theorem that
\[
y_1(t)=y_2(t) \quad \text{for  almost every $t\ge 0$.}
\]
We conclude that  $y_1=y_2$ in $Y_1$. Thus, $T_Z$ is well-defined and, by definition it is linear.

We now show that $T_Z$ is closed. Let $(x_n)_{n\in \N}$ be a sequence in $\mathcal D(T_Z)$ converging to $x\in Y_\infty^Z$ such that $y_n=T_Z x_n$ converges to $y\in Y_1$. Then, for $t\ge \tau \ge 0$ we have that
\[
x(t)-T(t,\tau)x(\tau) =\lim_{n\to \infty} (x_n(t)-T(t,\tau)x_n(\tau))
=\lim_{n\to \infty}\int_\tau^t  T(t,s)y_n(s)\, ds.
\]
On the other hand,  we have
\begin{align*}
\bigg{\lVert}\int_\tau^t   T(t,s)y_n(s)\, ds-\int_\tau^t   T(t,s)y(s)\, ds \bigg{\rVert} &\le M\int_\tau^t \lVert y_n(s)-y(s)\rVert\, ds\\
&\le M\int_\tau^t \lVert y_n(s)-y(s)\rVert_s \, ds\\
&\le M\lVert y_n-y\rVert_1,
\end{align*}
where $M=M(t,\tau)=\sup \{\lVert T(t,s)\rVert: s\in [\tau, t]\}$ is finite by the Banach-Steinhaus theorem. Since $y_n \to y$ in $Y_1$, we conclude that
\[
\lim_{n\to \infty} \int_\tau^t   T(t,s)y_n(s)\, ds=\int_\tau^t  T(t,s)y(s)\, ds,
\]
and therefore~\eqref{532} holds. We conclude that $x\in \mathcal D(T_Z)$ and $T_Zx=y$. Therefore, $T_Z$ is a closed linear operator. Similarly, one can show that $T_Z'$ is well-defined, linear  and closed.
\end{proof}

By the assumption in Theorem \ref{th.ad.dich}, the linear operators $T_Z$, $T_Z'$ are bijective, and by previous lemma and the Closed Graph Theorem they
have bounded inverse $G_Z \colon Y_1 \to Y_\infty^Z$ and $G_Z' \colon Y_\infty' \to Y_\infty^Z$, respectively.

For $\tau \ge 0$, set
\[
S(\tau)=\bigg{\{} v\in X: \sup_{t\ge \tau} \lVert T(t,\tau)v\rVert_t <\infty \bigg{\}} \quad \text{and} \quad U(\tau)=T(\tau, 0)Z.
\]
Clearly, $S(\tau)$ and $U(\tau)$ are subspaces of $X$ for each $\tau \ge 0$.
\begin{lemma}
For $\tau \ge 0$, we have that
\begin{equation}\label{1016}
X=S(\tau)\oplus U(\tau).
\end{equation}
\end{lemma}
\begin{proof}[Proof of lemma]
Let $\tau \ge 0$ and take $v \in X$. Set
$$g(s) = \chi_{[\tau,\tau+1]}(s)T (s, \tau)v,\;s\geq 0.$$
Clearly,
$g\in Y_1$. Since $T_Z$ is invertible, there exists $h\in \mathcal{ D}(T_Z)\subset Y_\infty^Z$ such that $T_Z h= g$. It follows
from~\eqref{532} that $$h(t) = T (t, \tau)(h(\tau) + v) \text{ for }t \ge \tau + 1.$$ Since $h \in Y_\infty$, we conclude that
$h(\tau)+v\in S(\tau)$.
Similarly, it follows from~\eqref{532} that $$h(\tau) = T (\tau, 0)h(0).$$ Since $h(0)\in Z$,
we have that $h(\tau) \in U(\tau)$ and thus $$v=(h(\tau)+v)+(-h(\tau))\in S(\tau)+U(\tau).$$
We have  proved that $X=S(\tau)+ U(\tau).$

Take now $v\in S(\tau)\cap U(\tau)$. Then, there exists $z\in Z$ such that $v=T(\tau, 0)z$. We consider a function $h\colon [0, \infty)\to X$, defined by
$$h(t)=T(t,0)z \text{  for } t\ge 0.$$
Clearly, $h\in Y_\infty^Z$. Since $h(t)=T(t,s)h(s)$ for all $t\geq s\geq 0$, it follows that $T_Zh=0$ and thus $h=0$. We conclude that $v=h(\tau)=0$, and  hence $S(\tau)\cap U(\tau)=\{0\}$. This completes the proof of the lemma.
\end{proof}

Let $P(\tau)\colon X \to S(\tau)$  and $Q(\tau)\colon X\to U(\tau)$ be the projections associated with the decomposition~\eqref{1016}, with $P(\tau)+Q(\tau)=\Id$.  Observe that~\eqref{pro} holds. Indeed, observe that
\[
T(t,\tau)S(\tau)\subset S(t) \quad \text{and} \quad T(t,\tau)U(\tau)\subset U(t), \quad \text{for $t\ge \tau \ge 0$.}
\]
Hence, we have that for every $x\in X$ and $t\ge \tau \ge 0$,
\[
P(t)T(t,\tau)x=P(t)T(t,\tau)P(\tau)x+P(t)T(t,\tau)Q(\tau)x=T(t,\tau)P(\tau)x.
\]
We conclude that~\eqref{pro} holds.
\begin{lemma}
For $t\ge \tau \ge 0$, the restriction $T(t,\tau)\rvert_{U(\tau)} \colon U(\tau)\to U(t)$ is invertible.
\end{lemma}

\begin{proof}[Proof of the lemma]
Let $t\geq \tau\geq 0$ and
take $x\in U(t)$. Then, there exists $z\in Z$ such that $x=T(t,0)z$. Since $T(\tau, 0)z\in U(\tau)$ and $x=T(t, \tau)T(\tau, 0)z$, we conclude that $T(t,\tau)\rvert_{U(\tau)}$ is surjective.

Let now $x\in U(\tau)$ such that $T(t,\tau)x=0$. Take $z\in Z$ such that $x=T(\tau, 0)z$. We define $u\colon [0, \infty)\to X$ by $u(s)=T(s,0)z$, $s\ge 0$.
Since $u(s)=0$ for $s\ge t$, we have that $u\in Y_\infty^Z$ and $T_Zu=0$.  Consequently,  $u=0$ and $x=u(\tau)=0$. This proves that $T(t,\tau)\rvert_{U(\tau)}$ is also injective. The proof of the lemma is completed.
\end{proof}

\begin{lemma}
There exists $M>0$ such that
\begin{equation}\label{7}
\lVert P(\tau)v\rVert_{\tau} \le M\lVert v\rVert_{\tau}, \text{ for all $v\in X$ and  $\tau \ge 0$.}
\end{equation}
\end{lemma}

\begin{proof}[Proof of the lemma]
Take $v\in X$ and  $\tau \ge 0$ . Moreover, given $h>0$, we define a function $g_h \colon [0,\infty) \to X$ by
\[
g_h(t)=\frac{1}{h}\chi_{[\tau,\tau+h]}(t)T(t,\tau)v.
\]
Clearly, $g_h\in Y_1$ and thus there exists $x_h\in \mathcal{D}(T_Z)$ such that $T_Zx_h=g_h$. We have
\[
\begin{split}
\lVert P(\tau)v\rVert_{\tau} & =\lVert x_h(\tau)+v\rVert_{\tau}
\le \lVert x_h(\tau)\rVert_{\tau}+\lVert v\rVert_{\tau}
\leq \lVert G_Zg_h\rVert_{\infty}+\lVert v\rVert_{\tau}.
\end{split}
\]
Moreover,
\[
\lVert G_Zg_h\rVert_{\infty}\le \lVert G_Z\rVert \cdot \lVert g_h\rVert_1 = \lVert G_Z\rVert \,\frac{1}{h}\int_{\tau}^{\tau+h}\lVert T(t,\tau)v\rVert_t \, dt.
\]
Letting $h\to 0$, we obtain
\[
\lVert P(\tau)v\rVert_{\tau} \le (1+\lVert G_Z\rVert)\,\lVert v\rVert_{\tau},
\]
and we conclude that~\eqref{7} holds for $M=1+\lVert G_Z\rVert$.
\end{proof}

\begin{lemma}\label{1047}
There exist constants $\lambda, D>0$ such that
\begin{equation}\label{8}
\lVert T(t,\tau)v\rVert_t \le De^{-\lambda (\rho(t)-\rho(\tau))}\lVert v \rVert_{\tau}, \text{ for $t\ge \tau \ge 0$ and $v\in S(\tau)$.}
\end{equation}
\end{lemma}
\begin{proof}[Proof of the lemma]
Fix $\tau\geq 0$ and  let $v\in S(\tau)$.  We consider the function
\[
u\colon [0,\infty) \to X,\quad
u(t)=\chi_{[\tau,\infty)}(t)T(t,\tau)v.
\]
Moreover, for any fixed $h> 0$, we define two functions $\phi_h \colon [0,\infty) \to \R$ and $g_h\colon [0, \infty) \to X$ by
\[
\phi_h(t)=
\begin{cases}
0, & 0\le t\le \tau, \\
\frac{1}{h}(t-\tau), & \tau \le t \le \tau +h, \\
1, & t\geq \tau +h,
\end{cases}
\]
and
\[
\quad g_h(t)=\frac{1}{h}\,\chi_{[\tau,\tau+h]}(t)\,T(t,\tau)v, \quad t\ge 0.
\]
It is easy to show that $g_h\in Y_1$,  $\phi_h u\in \mathcal{D}(T_Z)$  and $T_Z(\phi_h u)=g_h$. We have
\[
\begin{split}
\sup\limits_{t\geq \tau+h}\lVert u(t)\rVert_t&= \sup\limits_{t\geq \tau+h} \lVert \phi_h(t) u(t)\rVert_t
\le \lVert \phi_h u\rVert_{\infty}=\lVert G_Zg_h\rVert_{\infty}\\
& \le \lVert G_Z\rVert \cdot \lVert g_h\rVert_1 \\
&=\lVert G_Z\rVert \frac{1}{h} \int_{\tau}^{\tau+h}\lVert u(s)\rVert_s \, ds.
\end{split}
\]
Hence, letting $h\to 0$ we obtain the inequality
\[
\lVert u(t)\rVert_t \le \lVert G_Z\rVert \cdot \lVert v\rVert_{\tau},  \quad  \text{for every }  t\ge \tau.
\]
Thus,
\begin{equation}\label{1021}
\lVert T(t, \tau)v\rVert_t \le \lVert G_Z\rVert \cdot \lVert v\rVert_\tau, \quad    \text{for every }  t\ge \tau.
\end{equation}

Let us take $t\ge \tau$ and  $v\in S(\tau)$ such that $T(t, \tau)v\neq 0$, thus $T(s,\tau)v\neq 0$ for all $s\in[\tau,t]$. Let us consider $x, y\colon [0, \infty) \to X$  defined by
\[
y(s)=\chi_{[\tau,t]}(s)\,
\frac{T(s,\tau) v}{\lVert T(s,\tau)v\rVert_s}, \quad s\geq 0,
\]
and
\[
x(s)=\begin{cases}
0, & \text{ $0\le s \le \tau$,} \\
\int_\tau^s  \rho'(r) \frac{T(s, \tau)v}{\lVert T(r, \tau)v\rVert_r} \, dr, & \text{ $\tau < s \le t$,}\\
\int_\tau^t \rho'(r) \frac{T(s, \tau)v}{\lVert T(r, \tau)v\rVert_r}\, dr, & \text{ $s>t$.}
\end{cases}
\]
Note that $y\in Y_\infty'$ and $\lVert y\rvert_\infty'=1$. Furthermore, since $v\in S(\tau)$ we get that
$$\| x(s) \|_s\leq \int_\tau^t \frac{\rho'(r)}{\lVert T(r, \tau)v\rVert_r} \, dr \, \| T(s,\tau)v\|_s\leq a_{t,\tau,v} \,\sup\limits_{r\geq \tau} \|T(r,\tau)v\|_r<\infty,$$
for all $s\geq \tau$, where
$$a_{t,\tau,v}= \int_\tau^t \frac{\rho'(r)}{\lVert T(r, \tau)v\rVert_r} \, dr<\infty,$$
and thus $x\in Y_\infty^Z$.
It is  straightforward to show that $T_Z'x=y$. Consequently,
\[
\lVert x\rVert_\infty =\lVert G_Z' y\rVert_\infty \le \lVert G_Z'\rVert \cdot \lVert y\rVert_\infty' = \lVert G_Z'\rVert.
\]
Therefore,
\begin{equation}\label{1245}
\lVert G_Z'\rVert \ge \lVert x\rVert_\infty \ge \lVert x(t)\rVert_t = \lVert T(t, \tau)v\rVert_t \int_\tau^t  \frac{\rho'(r)}{\lVert T(r, \tau)v\rVert_r}\, dr.
\end{equation}
From~\eqref{1021} it follows that
$$\frac{1}{\lVert T(r, \tau)v\rVert_r}\geq \frac{1}{\|G_Z\|\cdot \|v\|_\tau}, \text{ for all } r\in [\tau,t],$$
and thus, from  \eqref{1245} we get
\begin{equation*}\label{eq.intermed}
\lVert G_Z'\rVert  \cdot \lVert G_Z\rVert \cdot \lVert v\rVert_\tau \ge  \lVert T(t, \tau)v\rVert_t  \, (\rho(t)-\rho(\tau)), \text{ for $t\geq \tau$ and $v\in S(\tau)$.}
\end{equation*}
Consequently,
$$(t-\tau) \left\|T\left(\rho^{-1}(t),\rho^{-1}(\tau)\right)v\right\|_{\rho^{-1}(t)}\leq \| G_Z'\|  \cdot \| G_Z\| \cdot \| v\|_{\rho^{-1}(\tau)}, $$
for $t\geq \tau$ and $v\in S\left(\rho^{-1}(\tau)\right)$.
Let $N_0\in\mathbb{N}^*$ such that $N_0>e \| G_Z'\|  \cdot \| G_Z\|$, and let $t\geq \tau+N_0$. Then,
\begin{align*}
N_0 \left\|T\left(\rho^{-1}(t),\rho^{-1}(\tau)\right)v\right\|_{\rho^{-1}(t)}&\leq (t-\tau) \left\|T\left(\rho^{-1}(t),\rho^{-1}(\tau)\right)v\right\|_{\rho^{-1}(t)}\\
&\leq \| G_Z'\|  \cdot \| G_Z\| \cdot \| v\|_{\rho^{-1}(\tau)},
\end{align*}
which implies that
there exists $N_0\in \N^*$ such that
\begin{equation}\label{1044}
\lVert T(\rho^{-1}(t), \rho^{-1}(\tau))v\rVert_{\rho^{-1}(t)} \le \frac{1}{e}\lVert v\rVert_{\rho^{-1}(\tau)},
\end{equation}
for  $t\ge \tau$ with  $t-\tau \ge N_0$ and $v\in S\left(\rho^{-1}(\tau)\right)$.
Take an arbitrary $t\ge \tau$ with $t-\tau\geq N_0$ and write $t-\tau$ in the form $$t-\tau=kN_0+r, \;k=k(t,\tau)\in \N^*\text{ and } r=r(t,\tau)\in[0,N_0).$$
Observing that
\[
\begin{split}
&T\left((\rho^{-1}(t), \rho^{-1}( \tau)\right) \\
&=T\left(\rho^{-1}(t), \rho^{-1}(\tau+kN_0)\right)\prod_{j=0}^{k-1} T\left( \rho^{-1}(\tau+(k-j)N_0), \rho^{-1}(\tau+(k-j-1)N_0)\right),
\end{split}
\]
it follows from~\eqref{1021} and~\eqref{1044} that
\[
\begin{split}
\lVert T\left(\rho^{-1}(t), \rho^{-1}( \tau)\right) v\rVert_{\rho^{-1}(t)} &\le \lVert G_Z\rVert \,e^{-k}\, \lVert v\rVert_{\rho^{-1}(\tau)} \\
&\le e\,\lVert G_Z\rVert \, e^{-\frac{1}{N_0} (t-\tau)}\,\lVert v\rVert_{\rho^{-1}(\tau)}, \\
\end{split}
\]
and thus~\eqref{8} holds with $\lambda=1/N_0$ and $D=e\,\lVert G_Z\rVert$. The proof of the lemma is completed.
\end{proof}
\begin{lemma}
There exist $\lambda, D>0$ such that
\begin{equation}\label{841}
\lVert T(t,\tau)v\rVert_t \le De^{-\lambda (\rho(\tau)-\rho(t))}\lVert v\rVert_\tau, \text{ for $0\le t\le \tau$ and $v\in U(\tau)$.}
\end{equation}
\end{lemma}

\begin{proof}[Proof of the lemma]
Take $\tau >0$ and $z\in Z$. We define a function $u\colon [0, \infty) \to X$ by
\[
u(t)=T(t,0)z, \quad \text{for $t\ge 0$.}
\]
For sufficiently small $h>0$, we define $\psi_h \colon [0, \infty)\to \R$,
\[
\psi_h(t)=\begin{cases}
1, & 0\le t\le \tau-h, \\
-\frac{t-\tau}{h}, & \tau-h \le t\le \tau, \\
0, & t\ge \tau.
\end{cases}
\]
Finally, we consider
\[
g_h\colon [0, \infty) \to X,\quad g_h=-\frac 1 h \chi_{[\tau-h, \tau]}\,u.
\]
It is easy to check that $g_h\in Y_1$, $\psi_h u\in \mathcal D(T_Z)$ and $T_Z(\psi_hu)=g_h$. Hence,
\[
\begin{split}
\sup\limits_{t\in [0, \tau-h ]}\lVert u(t)\rVert_t  &=\sup\limits_{t\in [0, \tau-h ]} \lVert  \psi_h(t) u(t)\rVert_t
\le \lVert \psi_hu\rVert_\infty
=\lVert G_Zg_h \rVert_\infty \\
&\le \lVert G_Z\rVert \cdot \lVert g_h \rVert_1 \\
&=\lVert G_Z\rVert \cdot \frac 1 h \int_{\tau-h}^\tau \lVert u(s)\rVert_s\, ds.
\end{split}
\]
Letting $h\to 0$, we get
\[
\lVert u(t)\rVert_t \le \lVert G_Z\rVert \cdot \lVert u(\tau)\rVert_{\tau}, \quad \text{for $0\le t\le \tau$,}
\]
which implies
\begin{equation}\label{921}
\lVert T(t,0)z\rVert_t \le \lVert G_Z\rVert \cdot \lVert T(\tau, 0)z\rVert_\tau, \quad \text{for $z\in Z$ and $0\le t\le \tau$.}
\end{equation}
Take now $z\in Z\setminus \{0\}$ and $0\le t\le \tau$. We define $x, y\colon [0, \infty) \to X$ by
\[
y(s)=\begin{cases}
-\frac{T(s, 0)z}{\lVert T(s, 0)z\rVert_s},  & \text{$0\le s\le \tau$,}\\
0, & \text{$s>\tau$,}
\end{cases}
\]
and
\[
x(s)=\begin{cases}
\int_s^\tau \rho'(r)\frac{T(s, 0)z}{\lVert T(r,0)z\rVert_r}\, dr, & \text{$0\le s\le \tau$,}\\
0, & \text{$s>\tau$.}
\end{cases}
\]
Observe that $y\in Y_\infty'$ and $\lVert y\rVert_\infty'=1$. Moreover, $x\in Y_\infty^Z$ and it is easy to check that $T_Z'x=y$. Hence,
\[
\lVert x\rVert_\infty=\lVert G_Z'y\rVert_\infty \le \lVert G_Z'\rVert.
\]
Consequently, for each $0\le s\le \tau$ we have
\[
 \lVert G_Z'\rVert \ge \lVert T(s, 0)z\rVert_s \int_s^\tau \rho'(r)\frac{1}{\lVert T(r,0)z\rVert_r}\, dr.
\]
Letting $\tau \to \infty$, we conclude that
\begin{equation}\label{1000}
\lVert G_Z'\rVert \ge \lVert T(s, 0)z\rVert_s \int_s^\infty \rho'(r)\frac{1}{\lVert T(r,0)z\rVert_r}\, dr \quad \text{for $s\ge 0$ and $z\in Z\setminus \{0\}$.}
\end{equation}
Take now $0\le t\le \tau$ and $z\in Z\setminus \{0\}$. It follows from~\eqref{921} and~\eqref{1000} that
\[
\begin{split}
\frac{1}{\lVert T(\rho^{-1}(t),0)z\rVert_{\rho^{-1}(t)}}  &\ge \frac{1}{\lVert G_Z'\rVert}\int_{\rho^{-1}(t)}^\infty \rho'(r)\frac{1}{\lVert T(r,0)z\rVert_r}\, dr \\
&\ge \frac{1}{\lVert G_Z'\rVert}\int_{\rho^{-1}(t)}^{\rho^{-1}(\tau)} \rho'(r)\frac{1}{\lVert T(r,0)z\rVert_r}\, dr \\
&\ge \frac{1}{\lVert G_Z'\rVert}\int_{\rho^{-1}(t)}^{\rho^{-1}(\tau)} \rho'(r)\frac{1}{\lVert G_Z\rVert \cdot \lVert T(\rho^{-1}(\tau), 0)z\rVert_{\rho^{-1}(\tau)}}\, dr \\
&=\frac{\tau-t}{\lVert G_Z'\rVert \cdot \lVert G_Z\rVert} \cdot \frac{1}{\lVert T(\rho^{-1}(\tau), 0)z\rVert_{\rho^{-1}(\tau)}}
\end{split}
\]
and thus
$$(\tau-t) \lVert T(\rho^{-1}(t),0)z\rVert_{\rho^{-1}(t)}\leq \lVert G_Z\rVert \cdot \lVert G_Z'\rVert \cdot \lVert T(\rho^{-1}(\tau), 0)z\rVert_{\rho^{-1}(\tau)}.$$
We conclude that there exists $N_0\in \N^*$ such that
\[
\lVert T(\rho^{-1}(t),0)z\rVert_{\rho^{-1}(t)} \le \frac{1}{e}\lVert T(\rho^{-1}(\tau), 0)z\rVert_{\rho^{-1}(\tau)},
\]
for $z\in Z$ and $0\le t\le \tau$ such that $\tau-t\ge N_0$.  Hence,
\[
\lVert T(\rho^{-1}(t), \rho^{-1}(\tau))v\rVert_{\rho^{-1}(t)}\le \frac{1}{e} \lVert v\rVert_{\rho^{-1}(\tau)},
\]
for $v\in U(\rho^{-1}(\tau))$ and $0\le t\le \tau$ such that $\tau-t\ge N_0$. By arguing as in the proof of Lemma~\ref{1047}, we find that there exist $\lambda, D>0$ such that
\[
\lVert T(\rho^{-1}(t), \rho^{-1}(\tau))v\rVert_{\rho^{-1}(t)}\le De^{-\lambda (\tau-t)} \lVert v\rVert_{\rho^{-1}(\tau)},
\]
for $v\in U(\rho^{-1}(\tau))$ and $0\le t\le \tau$, which readily implies the conclusion of the lemma.
\end{proof}
In order to complete the proof of the theorem, it is sufficient to observe that~\eqref{7}, \eqref{8} and~\eqref{841} imply that~\eqref{d1} and~\eqref{d2} hold.
\end{proof}

\begin{remark}
It is worth observing that in order to deduce the existence of a $\rho$-dichotomy we imposed two admissibility conditions. In the following two examples we will illustrate that this was necessary.
\end{remark}

\begin{example}
We consider an evolution family $\mathcal T=\{T(t,s)\}_{t\ge s\ge 0}$ given by
\[
T(t,s)=\Id, \quad t\ge s\ge 0.
\]
Furthermore, take $Z=\{0\}$ and let $\lVert \cdot \rVert_t=\lVert \cdot \rVert$ for $t\ge 0$. Then for each $y\in Y_1$, the unique $x\in Y_Z$ satisfying~\eqref{532} is given by
\[
x(t)=\int_0^t T(t,s) y(s)\, ds=\int_0^t  y(s)\, ds, \quad t\ge 0.
\]
Thus, the first assumption of Theorem~\ref{th.ad.dich} is fulfilled. On the other hand, $\mathcal T$ obviously doesn't admit  a $\rho$-dichotomy with respect to the family of norms $\lVert \cdot \rVert_t$, $t\ge 0$.
\end{example}
The following example is a simple modification of~\cite[Example 1]{Da.1}.
\begin{example}
 Let $X =\R$ with the standard Euclidean norm $\lvert \cdot \rvert$. Furthermore, let $\lVert \cdot \rVert_t=\lvert \cdot \rvert$ for $t\ge 0$ and take $Z=\{0\}$. Furthermore, let $\rho(t)=\ln (1+t)$ for $t\ge 0$.
We consider the sequence $(A_n)_{n\in \mathbb N}$ of operators on $X$ (which can of course be identified with numbers)  given by
\[
A_n=\begin{cases}
n & \text{if $n=2^l$ for some $l\in \mathbb N$,}\\
0 & \text{otherwise.}
\end{cases}
\]
Furthermore, for $t\ge s\ge 0$ we define
\[
T(t,s)=\begin{cases}
A_{\lfloor t\rfloor -1}\cdots A_{\lfloor s \rfloor}, & \text{$\lfloor t \rfloor \ge \lfloor s\rfloor+1$},\\
1, & \text{$\lfloor t\rfloor=\lfloor s\rfloor$.}
\end{cases}
\]
Then,  $\mathcal T=\{T(t,s)\}_{t\ge s\ge 0}$ is an evolution family.
By arguing as in~\cite[Example 1]{Da.1}, it is easy to check that the second assumption of Theorem~\ref{th.ad.dich} is satisfied and $\mathcal T$ doesn't admit a  $\rho$-dichotomy with respect to the family of norms $\lVert \cdot \rVert_t$, $t\ge 0$.
\end{example}

\section{Robustness of generalized dichotomies}\label{R}

In this section we apply our main results to prove that the concept of $\rho$-dichotomy with respect to a family $\{\lVert \cdot \rVert_t\}_{t\ge 0}$  of norms on $X$  persist under sufficiently small linear perturbations. As a consequence, we establish the robustness property of $\rho$-nonuniform exponential dichotomy.

\begin{theorem}\label{rob.1}
Assume that the evolution family $\{T(t,s)\}_{t\ge s \ge 0}$ admits a $\rho$-dichotomy with respect to a family $\{\lVert \cdot \rVert_t\}_{t\ge 0}$ of norms on $X$ satisfying
\[
\lVert x\rVert \le \lVert x\rVert_t \le Ce^{\epsilon\rho(t)} \lVert x\rVert, \quad \text{for $x\in X$ and $t\ge 0$,}
\]
for some $C>0$ and $\epsilon\geq 0$, such that
the mapping $t\mapsto \lVert x\rVert_t$ is continuous for each $x\in X$.
If $B:[0,\infty)\to \mathcal{B}(X)$ is a strongly continuous operator-valued function such that
\begin{equation}\label{eq.rob}
\|B(t)\|\leq \delta e^{-(\epsilon+a) \rho(t)} \rho'(t), \quad t\geq 0,
\end{equation}
for some $a>0$ and sufficiently small $\delta>0$, then the perturbed evolution family $\{U(t,s)\}_{t\ge s \ge 0}$ satisfying
\begin{equation}\label{eq.perturbed}
U(t,s)=T(t,s)+\int_s^t T(t,\tau)B(\tau)U(\tau,s)\,d\tau,\quad t\geq s\geq 0,
\end{equation}
admits a $\rho$-dichotomy with respect to the family of norms $\lVert \cdot \rVert_t$, $t\geq 0$.
\end{theorem}

\begin{proof}
Since $\{T(t,s)\}_{t\ge s \ge 0}$ admits a $\rho$-dichotomy with respect to the family  of norms $\lVert \cdot \rVert_t$, $t\ge 0$, it follows from Proposition \ref{p1} and Proposition \ref{p2} that there exists a closed subspace $Z\subset X$ such that the operators
$$T_Z \colon \mathcal D(T_Z)\subset Y_\infty^Z \to Y_1 \text{ and } T_Z' \colon \mathcal D(T_Z')\subset Y_\infty^Z \to Y_\infty',$$
defined in the proof of Theorem \ref{th.ad.dich}, are invertible and closed.
We consider the graph norms:
$$\|x\|_{T_Z}:=\|x\|_\infty+\|T_Zx\|_1, \quad \; x\in \mathcal D(T_Z),$$
and
$$\;\|x\|_{T_Z'}:=\|x\|_\infty+\|T_Z'x\|_\infty', \quad  x\in \mathcal D(T_Z').$$

Since $T_Z$, $T_Z'$ are closed, it follows that $\left( \mathcal D(T_Z),\|\cdot\|_{T_Z}\right)$, $\left( \mathcal D(T_Z'),\|\cdot\|_{T_Z'}\right)$ are Banach spaces. Furthermore,
$$T_Z:\left( \mathcal D(T_Z),\|\cdot\|_{T_Z}\right)\to \left( Y_1, \|\cdot\|_1 \right)$$
and
$$\quad T_Z':\left( \mathcal D(T_Z'),\|\cdot\|_{T_Z'}\right)\to \left( Y_\infty', \|\cdot\|_\infty' \right)$$
are bounded linear operators,  denoted simply by $T_Z$ and $T_Z'$, respectively.

We consider the linear operators $D:\mathcal D(T_Z)\to Y_1$, $D':\mathcal D(T_Z')\to Y_\infty'$ defined by
$$(Dx)(t)=B(t)x(t) \text{ and }
 (D'x)(t)=\frac{1}{\rho'(t)}B(t)x(t),\text{ for }t\geq 0.$$
One can easy check that these operators are well-defined. Furthermore, for each $x\in \mathcal D(T_Z)$ we have
\begin{align*}
\| Dx\|_1&=\int_0^\infty \|B(t)x(t)\|_t\,dt\\
&\leq C\int_0^\infty e^{\epsilon \rho(t)} \|B(t)x(t)\|\,dt\\
&\leq \delta C \int_0^\infty e^{-a  \rho(t)}\,\rho'(t) \, \|x(t)\|\,dt\\
&\leq \frac{\delta C}{a}\, \|x\|_\infty,
\end{align*}
and thus
\begin{equation}\label{eq.r1}
\| Dx\|_1\leq \frac{\delta C}{a}\, \|x\|_{T_Z},\quad x\in \mathcal D(T_Z).
\end{equation}
On the other hand, for $x\in \mathcal D(T_Z')$ we get
\begin{align*}
\|(D'x)(t)\|_t &=\frac{1}{\rho'(t)}\|B(t)x(t)\|_t\\
&\leq \frac{1}{\rho'(t)} C e^{\epsilon \rho(t)}\|B(t)x(t)\|\\
&\leq \delta \,C  e^{-a \rho(t)} \|x(t)\|\\
&\leq \delta \, C\,\|x\|_{T_Z'},
\end{align*}
for all $t\geq 0$, hence
\begin{equation}\label{eq.r2}
\|D'x\|_\infty'\leq \delta\,C\,  \|x\|_{T_Z'}, \quad x\in \mathcal D(T_Z').
\end{equation}

We define now the linear operators
$$U_Z:  \mathcal D(U_Z) \to Y_1, \quad U_Zx=y,$$
where  $\mathcal D(U_Z)$ is the set of all functions  $x\in Y_\infty^Z$ such that there exists $y\in Y_1$ satisfying
$$x(t)=U(t,s)x(s)+\int_s^t U(t, \tau)y(\tau)\, d\tau, \quad \text{for $t\ge s\ge 0$,}$$
and respectively,
$$U_Z':  \mathcal D(U_Z') \to Y_\infty', \quad U_Z'x=y,$$
where  $\mathcal D(U_Z')$ is the set of all functions  $x\in Y_\infty^Z$ such that there exists $y\in Y_\infty'$ satisfying
$$x(t)=U(t,s)x(s)+\int_s^t\rho'(\tau) U(t, \tau)y(\tau)\, d\tau, \quad \text{for $t\ge s\ge 0$.}$$

\begin{lemma}\label{lem.robust}
We have:
\begin{equation}\label{eq.r.3}
\mathcal D(T_Z)=\mathcal D(U_Z) \text{ and } T_Z=U_Z+D,
\end{equation}
and respectively,
\begin{equation}\label{eq.r.4}
\;\mathcal D(T_Z')=\mathcal D(U_Z') \text{ and } T_Z'=U_Z'+D'.
\end{equation}
\end{lemma}
\begin{proof}[Proof of the lemma]
Take $x\in \mathcal D(U_Z)$, that is $x\in Y_\infty^Z$ such that there exists $y\in Y_1$ with $U_Zx=y$. Then, for $t\geq s\geq 0$ we have
\begin{align*}
x(t)&=U(t,s)x(s)+\int_s^t U(t, \tau)y(\tau)\, d\tau\\
&= T(t,s)x(s)+\int_s^t T(t, \tau)B(\tau)U(\tau,s)x(s)\, d\tau+\int_s^t T(t, \tau)y(\tau)\, d\tau\\
&\quad+\int_s^t\int_\tau^t T(t,r)B(r)U(r,\tau)y(\tau)\,dr\,d\tau\\
&= T(t,s)x(s)+\int_s^t T(t,r)y(r)\, dr+\int_s^t T(t, r)B(r)U(r,s)x(s)\, dr\\
&\quad+\int_s^t\int_s^r T(t,r)B(r)U(r,\tau)y(\tau)\,d\tau\,dr\\
&= T(t,s)x(s)+\int_s^t T(t,r)\left(y(r)+B(r)x(r)\right)\,dr,
\end{align*}
thus $x\in \mathcal D(T_Z)$ and $$T_Zx=y+Dx=(U_Z+D)x.$$
Reversing the arguments, we conclude that  \eqref{eq.r.3} holds. Similarly, one can prove \eqref{eq.r.4}.
\end{proof}

Now, we continue the proof of the theorem.
From \eqref{eq.r.3} and \eqref{eq.r1} we have
$$ \| (U_Z-T_Z)x\|_1=\|Dx\|_1\leq  \frac{\delta C}{a}\, \|x\|_{T_Z},\text{ for all } x\in \mathcal D(T_Z)=\mathcal D(U_Z),$$
which implies that $U_Z:  \mathcal D(U_Z) \to Y_1$ is bounded. Since $T_Z$ is invertible, we obtain that $U_Z$ is also invertible for sufficiently small $\delta>0$.
Similarly, one can show that $U_Z'$ is  invertible for sufficiently small $\delta>0$. By Theorem \ref{th.ad.dich} we conclude that the perturbed evolution family
$\{U(t,s)\}_{t\ge s \ge 0}$  admits a $\rho$-dichotomy with respect to the family of norms $\|\cdot\|_t$, $t\geq 0$.
\end{proof}

From Proposition \ref{prop.equiv} and Theorem \ref{rob.1} we are able now to establish the robustness property of $\rho$-nonuniform exponential dichotomy.
\begin{corollary}\label{cor.robust}
Assume that  $\mathcal{T}=\{T(t,s)\}_{t\ge s \ge 0}$ admits a $\rho$-nonuniform exponential dichotomy. If
$B:[0,\infty)\to \mathcal{B}(X)$ is a strongly continuous operator-valued function satisfying \eqref{eq.rob}
for some $a>0$ and sufficiently small $\delta>0$, then the perturbed evolution family satisfying \eqref{eq.perturbed}
admits also a $\rho$-nonuniform exponential dichotomy.
\end{corollary}

\begin{remark}
We stress that the robustness of  $\rho$-nonuniform exponential dichotomies was established in~\cite[Theorem 1]{Ba.Va.3} using different techniques. However, we point out  that we establish robustness under a wider class of perturbations than those considered in~\cite[Theorem 1]{Ba.Va.3}. On the other hand,  we consider a smaller class of rate functions $\rho$.
\end{remark}

\section{Acknowledgements}
D. D. was supported in part
by Croatian Science Foundation under the project IP-2019-04-1239 and by
the University of Rijeka under the projects uniri-prirod-18-9 and uniri-prprirod-19-16. 
N. L. was supported  by the research grant GNaC2018-ARUT, no. 1360/01.02.2019.

\end{document}